\documentclass[12pt,a4paper]{amsart}

\title[Regularizations of residue currents]{Regularizations of residue currents}

\author{Jan-Erik Bj\"{o}rk \& H\aa kan Samuelsson}
\thanks{The second author was partially supported by a Post Doctoral Fellowship from the Swedish Research Council.}
\address{J.-E. Bj\"{o}rk, Department of Mathematics, Stockholm University, SE-106 91 Stockholm, Sweden}
\email{jeb@math.su.se}
\address{H. Samuelsson, Department of Mathematical Sciences, Division of Mathematics, University of Gothenburg and 
Chalmers University of Technology, SE-412 96 G\"{o}teborg, Sweden}
\email{hasam@math.chalmers.se}

\usepackage[english]{babel}
\usepackage{amsmath}
\usepackage{amsthm,amssymb,latexsym}
\usepackage{t1enc}
\usepackage{fancyheadings}
\usepackage{mathrsfs}
\usepackage{graphicx}

\newtheorem{proposition}{Proposition}
\newtheorem{theorem}[proposition]{Theorem}
\newtheorem{lemma}[proposition]{Lemma}
\newtheorem{corollary}[proposition]{Corollary}

\theoremstyle{definition}
\newtheorem{definition}[proposition]{Definition}

\newtheorem{remark}[proposition]{Remark}

\newcommand{\C}{\mathbb{C}}
\newcommand{\debar}{\bar{\partial}}

\newcommand{\CH}{\textrm{{\bf CH}}}
\newcommand{\B}{\mathbb{B}}

\def\newop#1{\expandafter\def\csname #1\endcsname{\mathop{\rm #1}\nolimits}}
\newop{span}

\begin{document}
\nocite{*}
\bibliographystyle{plain}

\begin{abstract}
Under assumptions about complete intersection, we prove that Coleff-Herrera type currents
satisfy a robust calculus in the sense that natural regularizations of such currents can be 
multiplied to yield regularizations of the Coleff-Herrera product of the currents.
\end{abstract}

\maketitle
\thispagestyle{empty}

\section{Introduction}
Let $f$ be a holomorphic function defined on the unit ball $\B\subset \C^n$. Then $1/f$ exists as a 
principal value distribution, or rather as a $(0,0)$-current, on $\B$, i.e., 
\begin{equation*}
\lim_{\epsilon \to 0^+}\int_{\{|f|>\epsilon\}}\varphi/f
\end{equation*}
exists for $\varphi\in \mathscr{D}_{n,n}(\B)$ and defines a continuous functional on $\mathscr{D}_{n,n}(\B)$.
This was first proved by Herrera-Lieberman, \cite{HL}, using Hironaka's theorem on resolutions of singularities. In fact,
by Hironaka's theorem one may assume that $f$ is a monomial, and in that case it is possible to compute the limit by
hand. The proof also shows that one may take the limit of integrals over $\{|\tilde{f}|>\epsilon\}$, where
$\tilde{f}$ is any holomorphic function such that $\tilde{f}^{-1}(0)\supseteq f^{-1}(0)$.
The current $1/f$ is obviously closely related to division problems; if $h$ is holomorphic then $h/f$ is at least a 
current, and it is holomorphic if and only if it is $\debar$-closed, i.e., if and only if $0=\debar(h/f)=h\debar(1/f)$.
Hence, $h$ is in the ideal, $\langle f\rangle$, generated by $f$ if and only if $h$ annihilates 
the current $\debar(1/f)$. This current 
clearly has support on $Z_f=f^{-1}(0)$ and it is related to Lelong's
integration current $[Z_f]$, see \cite{Lelong}, by the Poincar\`{e}-Lelong formula: $2\pi i [Z_f]=\debar (1/f)\wedge df$.
The current $\debar (1/f)$ is called the residue current associated to $f$ and it is thus 
an analytic object that describes the algebraic-geometric object $\langle f\rangle$. 

Now, let $V$ be a pure $n$-dimensional analytic subset of a complex $N$-dimensional 
manifold $X$ and let $f\colon X\to \C$ be a holomorphic
function such that $V\setminus f^{-1}(0)$ is a dense subset of $V_{reg}$. Then the principal value of $1/f$ exists on 
$V$, i.e.,
\begin{equation}\label{hV}
\lim_{\epsilon \to 0^+}\int_{V\cap\{|\tilde{f}|>\epsilon\}} \varphi /f
\end{equation} 
exists for $\varphi\in \mathscr{D}_{n,n}(X)$ and holomorphic $\tilde{f}$ with $\tilde{f}^{-1}(0)\supseteq f^{-1}(0)$
and yields a well defined current denoted $(1/f)[V]$.
The existence of this limit follows from the case $V=\B$ by
Hironaka's theorem. A sheaf of currents on $X$ supported on $V$ is then 
obtained by applying holomorphic differential operators to such currents. This sheaf is (equivalent to)
the sheaf $\CH_V[*S]$, see Definition \ref{CHdef}, and it is this kind of currents we will consider in this paper.
The kernel of $\debar$ in $\CH_V[*S]$ is denoted $\CH_V$ and is actually sufficiently ample  
to represent moderate cohomology in the sense that $\CH_V \simeq \mathcal{H}_{[V]}^P(\mathcal{O}_X)$,
see \cite{DS}; here $P=N-n$ is the codimension of $V$. The notation $\CH$ refers to Coleff-Herrera type currents.

\smallskip

Let us return to the case $V=\B\subset \C^n$ and consider a holomorphic mapping
$f=(f_1,\ldots,f_p)\colon \B \to \C^p$. To find a current that describes the ideal, $\langle f\rangle$,
generated by $f_1,\ldots,f_p$ it is tempting to try to define the product 
$\debar (1/f_1)\wedge \cdots \wedge \debar (1/f_p)$. If $f$ defines a complete intersection, i.e., 
$f^{-1}(0)$ has codimension $p$, it is possible to give a well defined meaning to this product. 
This was first done by Coleff-Herrera, \cite{CH}, as follows.
Let $\varphi \in \mathscr{D}_{n,n-p}(\B)$ and put
\begin{equation*}
I_f^{\varphi}(\epsilon):=\int_{\cap \{|f_j|^2=\epsilon_j\}} \varphi /(f_1\cdots f_p).
\end{equation*}
Coleff-Herrera proved that the limit of $I_f^{\varphi}(\epsilon)$ as $\epsilon \to 0$
along any ``admissible path'' exists and defines a current of bidegree $(0,p)$, denoted
$\debar (1/f_1)\wedge \cdots \wedge \debar (1/f_p)$ or $R^f$ for short, that is alternating
with respect to the ordering of the tuple $f$.
Admissible path here means that $\epsilon \to 0$ along a path in the first orthant such that
$\epsilon_j/\epsilon_{j-1}^{k}\rightarrow 0$ for $j=2,\ldots,p$ and all
$k\in \mathbb{N}$. It was later proved by Dickenstein-Sessa, \cite{DS}, and Passare, \cite{PDr}, 
independently that $R^f$
describes the ideal $\langle f\rangle$ in the sense that its annihilator is precisely $\langle f\rangle$.
We remark that even if $f$ does not define a complete intersection, the limit of 
$I_f^{\varphi}(\epsilon)$ along an admissible path exists but does not yield a well defined current
associated to $f$ as one easily sees from the simple case $f_1=z^2$, $f_2=zw$. Currents describing
general ideals have recently been defined by Andersson-Wulcan, \cite{AW}; see also 
Section \ref{bevissektion} below.

From now on we stick to the (generic) case that $f$ defines a complete intersection.
The first question raised by Coleff-Herrera in \cite{CH} is whether it is necessary to take limits along 
admissible paths or not. It turned out to be necessary; Passare-Tsikh, \cite{PTex}, showed that if 
$f=(z^4,z^2+w^2+z^3)$ then $I_f^{\varphi}(\epsilon)$ does not have an unrestricted limit 
as $\epsilon\to 0$ (for all $\varphi$). A generic family of examples with this property was later found
by the first author; even examples with $I_f^{\varphi}(\epsilon)\to \infty$ along certain paths are
constructed, see, e.g., \cite{JEBabel}; see also Pavlova, \cite{pavlova}. 
However, our main theorem implies that the mild average of 
$I_f^{\varphi}(\epsilon)$,
\begin{equation*}
\mathcal{I}_f^{\varphi}(\epsilon):=
\int_{s\in [0,\infty)^p} I_f^{\varphi}(s) \,d\chi_1(s_1/\epsilon_1)\wedge \cdots
\wedge d\chi_p(s_p/\epsilon_p),
\end{equation*}
where $\chi_j\in C^{\infty}([0,\infty])$, $\chi_j(0)=0$, and $\chi_j(\infty)=1$, depends H\"{o}lder
continuously on $\epsilon \in [0,\infty)^p$ and tends to the Coleff-Herrera product $R^f$ as $\epsilon\to 0$.
In fact, we prove
\begin{theorem}\label{huvudsats}
Let $X$ be a complex $N$-dimensional manifold, $V\subseteq X$ an analytic subset of pure dimension $n$, and 
$f=(f_1,\ldots,f_q)\colon X\to \C^q$ a holomorphic mapping such that $(f_1,\ldots,f_p,f_j)$ locally
defines a complete intersection on $V$ for $p+1 \leq j \leq q$. 
Let also $\chi_j$, $1\leq j \leq q$, be smooth on $[0,\infty]$, vanish to order $\ell_j$ at $0$ 
and $\chi_j(\infty)=1$.
Then for any $\mu \in \CH_V$ and $\varphi\in \mathscr{D}_{N,n-p}(X)$ we have
\begin{equation*}
\Big| \frac{\debar \chi_1^{\epsilon}\wedge \cdots \wedge \debar \chi_p^{\epsilon}
\chi_{p+1}^{\epsilon}\cdots \chi_q^{\epsilon}}{f_1^{\ell_1}\cdots f_q^{\ell_q}}\wedge \mu. \varphi \hspace{5cm}
\end{equation*} 
\begin{equation*}
\hspace{3cm} -\debar\frac{1}{f_1^{\ell_1}}\wedge \cdots \wedge \debar\frac{1}{f_p^{\ell_p}}
\frac{1}{f_{p+1}^{\ell_{p+1}}\cdots f_q^{\ell_q}}
\wedge \mu. \varphi \Big| \leq C \|\varphi\|_M \epsilon^{\omega},
\end{equation*}
where $\chi_j^{\epsilon}=\chi_j(|f_j|^2/\epsilon_j)$, and $M$ and $\omega$ are positive constants that
only depend on $f$ and $\textrm{Supp}(\varphi)$, while the positive constant $C$ also might depend
on the $C^M$-norm of the $\chi_j$-functions.
\end{theorem}



Let $V=X$, $\mu=1$, $\ell_j=1$ and let the $\chi_j$ be smooth regularizations of the characteristic function 
of $[1,\infty)$; by this we mean that $\chi_j$ is a smooth increasing function on $[0,\infty)$
that is $0$ close to $0$ and $1$ close to $\infty$. The theorem implies that the smooth form
\begin{equation}\label{passareform}
\debar\frac{\chi_1^{\epsilon}}{f_1}\wedge \cdots \wedge \debar\frac{\chi_p^{\epsilon}}{f_p}\cdot
\frac{\chi_{p+1}^{\epsilon}}{f_{p+1}} \cdots \frac{\chi_q^{\epsilon}}{f_q}
\end{equation}
converges unrestrictedly to the mixed residue and principal value current
\begin{equation}\label{CHP}
\debar \frac{1}{f_1} \wedge \cdots \wedge \debar\frac{1}{f_p}\cdot \frac{1}{f_{p+1}}\cdots \frac{1}{f_q}
\end{equation}
introduced by Coleff-Herrera, \cite{CH}, and Passare, \cite{PCrelle}.
It is proved in \cite{PCrelle} that if $\epsilon_j=\delta^{s_j}$, then \eqref{passareform}
has a limit, independent of $s=(s_1,\ldots,s_q)\in \mathbb{R}_{+}^q$, as $\delta \to 0^+$ as long as $s$ avoids finitely
many hyperplanes $H_a=\{t\in \mathbb{R}_{+}^q; t\cdot a=0\}$, $a\in \mathbb{N}^q$; we say that $\epsilon \to 0$
inside a Passare sector. Our result is thus a sharpening and a generalization of Passare's result and shows that 
there is a robust calculus for Coleff-Herrera type currents. In particular, we have the appealing
formula
\begin{equation*}
R^f. \varphi=\lim_{\epsilon\to 0} \int_X
\debar\frac{\bar{f}_1}{|f_1|^2+\epsilon_1}\wedge \cdots \wedge \debar\frac{\bar{f}_p}{|f_p|^2+\epsilon_p}\wedge \varphi,
\quad \varphi \in \mathscr{D}_{N,N-p}(X),
\end{equation*}
which follows by taking $\chi_j(t)=t/(t+1)$.

\bigskip

Another approach to the Coleff-Herrera product, $R^f$, is based on analytic continuation of currents,
a technique with roots in the works of Atiyah, \cite{Atiyah}, and Gelfand-Shilov, \cite{GS}. 
In the context of residue currents, it has been developed
by several authors, e.g., Barlet-Maire, \cite{BM}, Yger, \cite{Y}, Passare-Tsikh, \cite{PTcanada}, Berenstein-Gay-Yger, 
\cite{BGY}, and by the second author in the recent paper \cite{hasamArkiv}. 
Computing the Mellin transform of the integral in \eqref{hV} (considered as a function of $\epsilon$)
one obtains
\begin{equation}\label{analytisk1}
\int_V |\tilde{f}|^{2\lambda} \varphi /f
\end{equation}
if $\mathfrak{Re}\, \lambda >> 1$. One can show, either by using a Bernstein-Sato functional equation or by
computing directly in a resolution of $V$ where $f^{-1}(0)$ has normal crossings, that \eqref{analytisk1} (as a 
function of $\lambda$) has a meromorphic continuation to all of $\C$ and that its poles are contained in 
an arithmetic progression $\{-s-\mathbb{N}\}$, $s\in \mathbb{Q}_+$. It is thus analytic in a neighborhood of 
the origin, and moreover, its value there defines the action of a current. This current is the current 
$(1/f)[V]$, as one easily shows in a resolution.  

The Coleff-Herrera-Passare current \eqref{CHP} can be obtained in a similar manner; consider the function
\begin{equation}\label{analytisk2}
\lambda \mapsto
\int_X \frac{\debar |f_1|^{2\lambda_1} \wedge \cdots \wedge \debar |f_p|^{2\lambda_p}
|f_{p+1}|^{2\lambda_{p+1}}\cdots |f_q|^{2\lambda_q}}{f_1\cdots f_q}\wedge \varphi,
\end{equation} 
where $f$ defines a complete intersection on $X$ and $\mathfrak{Re}\, \lambda_j >> 1$. 
One can similarly show that it has a meromorphic extension to 
$\C^{q}$. It was recently showed by the second author in \cite{hasamArkiv} that it actually is
analytic in a neighborhood of $\cap_j \{\mathfrak{Re}\, \lambda_j \geq 0\}$. By results of Yger, 
it was known before that the restriction of \eqref{analytisk2} 
to any complex line of the form 
$\{\lambda=(t_1z,\ldots,t_qz); z\in \C\}$, $t_j\in \mathbb{R}_+$, has an analytic continuation to a neighborhood
containing the origin and that the value there
equals \eqref{CHP}. Moreover, it was also known that if $q=2$, then \eqref{analytisk2} 
has an analytic continuation to a neighborhood of the origin in $\C^2$; see, e.g., \cite{BGVY}, \cite{BGY} 
for proofs. Event though not explicitly
stated in \cite{hasamArkiv}, we remark that it follows from the proof that one may replace $X$ in 
\eqref{analytisk2} by a pure dimensional analytic subset $V$ of $X$ and still have analyticity 
in a neighborhood of the origin.

We conclude the introduction with the simple but useful observation that
expressions like $\chi(|f|^2/\epsilon)/f^{\ell}$ essentially are invariant under holomorphic 
differential operators. More precisely, if $\chi\in C^{\infty}([0,\infty])$ and vanish to 
order $\ell$ at $0$, then 
\begin{equation}\label{invarians}
\frac{\partial}{\partial z_j}\frac{\chi(|f|^2/\epsilon)}{f^{\ell}}=
\frac{\partial f}{\partial z_j} \frac{\tilde{\chi}(|f|^2/\epsilon)}{f^{\ell+1}},
\end{equation}
where $\tilde{\chi}(t)=t\chi'(t)-\ell \chi(t)$ is smooth on $[0,\infty]$, vanishes to 
order $\ell +1$ at $0$, and $\chi(\infty)=-\ell \chi(\infty)$.

\section{The case of three functions}
We first note that $I_f^{\varphi}(\epsilon)$ might be discontinuous already when $f=(f_1,f_2)$ consists of two
functions as the Passare-Tsikh example shows. The technical reason is the presence of charts of resonance,
i.e., charts on the resolution manifold where it is not possible to choose coordinates so that 
the pullback of both $f_1$ and $f_2$ are monomials. To deal with the charts of resonance the smoothness
of the $\chi$-functions has to be used; we refer to \cite{hasamJFA} for the details.
In the case of three functions a new difficulty arise; it is no longer
a local problem on the resolution manifold to prove that 
\begin{equation}\label{exekv1}
\int \debar \frac{\bar{f}_1}{|f_1|^2+\epsilon_1}\wedge \debar \frac{\bar{f}_2}{|f_2|^2+\epsilon_2}
\frac{\bar{f}_3}{|f_3|^2+\epsilon_3}\wedge \varphi
\end{equation}
has an unrestricted limit. We illustrate this by considering a simple example; the example of 
Section 3 in \cite{hasamArkiv}. We let $f_1=z_1$, $f_2=z_2$, and $f_3=z_3$ in $\C^3$. Then, obviously,
\eqref{exekv1} has an unrestricted limit for all $(3,1)$-test forms $\varphi$ in $\C^3$. Now, we
let $\varphi=\phi dz\wedge d\bar{z}_3$, where $\phi$ is a test function, we blow up $\C^3$ along the 
$z_3$-axis, and we compute \eqref{exekv1} on the blow-up. The blow-up has two standard charts,
one of which is given by $(w_1,w_2,w_3)\mapsto (w_1,w_1w_2,w_3)=(z_1,z_2,z_3)$. Let us consider the 
contribution, $\mu^{\varphi}(\epsilon)$, to \eqref{exekv1} from this chart. One verifies easily,
using Cauchy's formula, that $\lim_{\epsilon_1 \to 0^+}\mu^{\varphi}(\epsilon)=0$ for fixed
$\epsilon_2,\epsilon_3>0$. On the other hand, one similarly shows that
\begin{equation*}
\lim_{\epsilon_3 \to 0^+}\lim_{\epsilon_1 \to 0^+}\lim_{\epsilon_2 \to 0^+}\mu^{\varphi}(\epsilon)=
-(2\pi i)^2\int_{z_3}\frac{\phi(0,0,z_3)}{z_3}dz_3\wedge d\bar{z}_3,
\end{equation*}
which clearly is non-zero for certain choices of $\phi$. Both the charts on the blow-up therefore have 
to be considered in order to see that \eqref{exekv1} has an unrestricted limit.

In general, however, what can be showed in each chart separately is that
\begin{equation*}
\Big|\int \debar \frac{\bar{f}_1}{|f_1|^2+\epsilon_1}\wedge \debar \frac{\bar{f}_2}{|f_2|^2+\epsilon_2}\wedge
\big( \frac{\bar{f}_3}{|f_3|^2+\epsilon_3}-\frac{1}{f_3}\big) \varphi\Big|
\leq C \epsilon_3^{\omega}
\end{equation*}
for positive constants $C$ and $\omega$ that do not depend on $\epsilon_1,\epsilon_2$, see
Proposition \ref{nyckelprop} below. To see that \eqref{exekv1} has an unrestricted limit it is therefore
enough to show that 
\begin{equation*}
\int \debar \frac{\bar{f}_1}{|f_1|^2+\epsilon_1}\wedge \debar \frac{\bar{f}_2}{|f_2|^2+\epsilon_2}\wedge
\frac{1}{f_3} \varphi =
\int \debar \frac{\bar{f}_1}{|f_1|^2+\epsilon_1} \frac{\bar{f}_2}{|f_2|^2+\epsilon_2}
\frac{1}{f_3} \wedge \debar \varphi
\end{equation*}
\begin{equation*}
\hspace{5.7cm} + \int \debar \frac{\bar{f}_1}{|f_1|^2+\epsilon_1} \frac{\bar{f}_2}{|f_2|^2+\epsilon_2}\wedge
\debar \frac{1}{f_3} \wedge \varphi,
\end{equation*}
has an unrestricted limit. But now we have only two parameters and, moreover, on the right hand side 
there is only $\debar$ in front of one of the parameter depending factors. With an appropriate
induction hypothesis and the result of Proposition \ref{nyckelprop} one can then conclude that 
\eqref{exekv1} has an unrestricted limit; see Section \ref{bevissektion} for details. 

\section{Coleff-Herrera currents}\label{CHsektion}
In this section we review the facts we will need about Coleff-Herrera type currents. The results are
well-known but for the readers convenience we supply detailed proofs.
Let $X$ be a complex $N$-dimensional manifold and let $V\subseteq X$ be a reduced subvariety of pure dimension $n$.
Put $P=N-n$ and let $\mathcal{J}_V$ be the ideal (sheaf) generated by $V$.
A (possibly singular) hypersurface $S\subset X$ is called $V$-polar if $V\setminus S$ is a dense subset of 
$V_{reg}$. Recall from the introduction, cf.\ \eqref{hV}, that if $h\in \mathcal{O}(X)$ and $h^{-1}(0)$ is $V$-polar,
then the principal value current $(1/h)[V]$ exists. It is often convenient to use the technique of analytic continuation
when working with this current; recall from the introduction that
if $\tilde{h}$ is any holomorphic function such that $\tilde{h}^{-1}(0)$ is $V$-polar
and contains $h^{-1}(0)$, then 
\begin{equation*}
\frac{1}{h}[V]. \varphi=
\int_{V} |\tilde{h}|^{2\lambda}\frac{\varphi}{h}\,\Big|_{\lambda=0}, 
\quad \varphi \in \mathscr{D}_{n,n}(X).
\end{equation*}
The next lemma shows, in particular, that this kind of currents have the Standard Extension Property;
a current $\mu$ has the Standard Extension Property with respect to a pure dimensional analytic set $V$
if for any holomorphic function $g$ such that $V\setminus g^{-1}(0)$ is dense in $V$ we have 
$\lim_{\epsilon \to 0^+}\chi(|g|^2/\epsilon)\mu = \mu$, where $\chi$ is a smooth regularization of the 
characteristic function of $[1,\infty)$. Two currents, $\mu$ and $\tilde{\mu}$, which both have the Standard 
Extension Property and are equal outside a hypersurface $H$, i.e., $\mu. \varphi = \tilde{\mu}. \varphi$ if 
$\varphi$ has support outside $H$, are thus equal. 

\begin{lemma}\label{seplemma}
Let $h,f\in \mathcal{O}(X)$ and assume that $h^{-1}(0)$ is $V$-polar and that $V\setminus f^{-1}(0)$ is dense
in $V$. If $\chi$ is a bounded function on $[0,\infty]$ that is identically $0$ close to $0$ and continuous
at $\infty$, then
\begin{equation*}
\lim_{\epsilon\to 0^+} \frac{\chi(|f|^2/\epsilon)}{f^{\ell}}\cdot
\frac{1}{h}[V] = \frac{\chi(\infty)}{f^{\ell}h}[V].
\end{equation*}
\end{lemma}
\begin{proof}
By Hironaka's theorem one may assume that $V$ is an $n$-dimensional manifold and that $\{h\cdot f=0\}$
has normal crossings. Locally one can then choose coordinates $x$ such that $f=x^{\alpha}$ and 
$h=vx^{\beta}$, where $v$ is an invertible holomorphic function. Letting $\varphi=\phi \,dx\wedge d\bar{x}$
we thus see that $\chi/(f^{\ell}h)[V].\, \varphi$ is a finite sum of terms like
\begin{equation}\label{sep1}
\int_V |x^{\beta}|^{2\lambda} \frac{\chi(|x^{\alpha}|^2/\epsilon)}{x^{\ell\alpha +\beta}}
\frac{|v|^{2\lambda}}{v}\phi \,dx\wedge d\bar{x}\, \Big|_{\lambda=0}.
\end{equation}
By Lemma 6 in \cite{hasamJFA} we can write
\begin{equation*}
\frac{|v|^{2\lambda}}{v}\phi =
\sum_{K+L<\ell \alpha+\beta -{\bf 1}} x^K\bar{x}^L \Phi_{K,L}(\lambda,x)
+\sum_{K+L=\ell \alpha+\beta -{\bf 1}} x^K\bar{x}^L \Phi_{K,L}(\lambda,x),
\end{equation*}
where each $\Phi_{K,L}$, with $K+L < \ell\alpha +\beta- {\bf 1}$, is independent of at least some coordinate $x_j$.
Using this, and changing to polar coordinates, one readily checks that the first sum on the right hand side
does not contribute to the integral \eqref{sep1}. Substituting the second sum into \eqref{sep1} the singularity 
of the integrand vanishes and one may put $\lambda=0$ and let $\epsilon\to 0^+$ to obtain
\begin{equation*}
\lim_{\epsilon\to 0^+}
\int_{r,\theta} \chi(r^{2\alpha}/\epsilon)
\sum_{K+L=\ell \alpha+\beta -{\bf 1}}e^{i\theta\cdot (K-L-\ell\alpha -\beta)} \Phi_{K,L}(0,re^{i\theta})dr d\theta
\end{equation*}
\begin{equation*}
=\chi(\infty)\int_{r,\theta} 
\sum_{K+L=\ell \alpha+\beta -{\bf 1}}e^{i\theta\cdot (K-L-\ell\alpha -\beta)} \Phi_{K,L}(0,re^{i\theta})dr d\theta.
\end{equation*}
Computing $(1/(f^{\ell}h))[V]. \varphi$ in the same way, using the same desingularization and choice of coordinates one 
easily checks that this last integral is what one gets (in the $x$-chart). 
\end{proof}

Let $Q$ be a holomorphic differential operator in $X$ and put $\mu=(1/h)[V]$. It is clear that 
$\bar{\mathcal{J}}_V\cdot Q(\mu) =0$ and that $\textrm{Supp}(\debar Q(\mu)) \subseteq h^{-1}(0)$.
Moreover, from the lemma it follows that $Q(\mu)$ has the Standard Extension Property. In fact,
let $\{f=0\}$ be a hypersurface in $X$ such that $V\setminus \{f=0\}$ is dense in $V$, let $\chi$
be a smooth regularization of the characteristic function of $[1,\infty)$, and put
$\chi^{\epsilon}=\chi(|f|^2/\epsilon)$. Then a simple computation shows that
\begin{equation}\label{Qchi}
\chi^{\epsilon}\cdot Q(\mu)=Q(\chi^{\epsilon}\cdot \mu)+\sum_j Q_j\big(\frac{\chi_j^{\epsilon}}{f^{k_j}}\cdot \mu\big),
\end{equation}
where $Q_j$ are certain differential operators and $\chi_j^{\epsilon}=\chi_j(|f|^2/\epsilon)$
with $\chi_j$ smooth on $[0,\infty]$ and $\chi_j(\infty)=0$; cf.\ \eqref{invarians}. From the lemma it then follows
that $\chi^{\epsilon}\cdot Q(\mu) \to Q(\mu)$.

With these facts in mind we define the Coleff-Herrera currents on $V$, $\CH_V$, 
and the Coleff-Herrera currents on $V$ with pole
along $S$, $\CH_V[*S]$, for hypersurfaces $S\subseteq X$ such that $V\setminus S$ is dense in $V$.

\begin{definition}\label{CHdef}(The sheaves $\CH_V$ and $\CH_V[*S]$.)

A current $\mu$ of bidegree $(0,P)$ on an open set $U\subseteq X$ is a section of 
$\CH_V$ over $U$ if 
\begin{itemize}
\item[1)] $\mu$ has the Standard Extension Property,
\item[2)] $\bar{\mathcal{J}}_V\cdot \mu =0$,
\item[3)] $\debar\mu =0$.
\end{itemize}
If $\mu$ satisfies 1), 2), and $\textrm{Supp}(\debar \mu) \subseteq S$, then
we say that $\mu$ is a section of $\CH_V[*S]$ over $U$.
\end{definition}

We have the following local representation of Dolbeault-Lelong type of currents in $\CH_V$,
and consequently of currents in $\CH_V[*S]$; see below. The slick proof is taken from \cite{matsa2}.

\begin{proposition}\label{DLprop}
Let $X$ be a neighborhood of the closure of the unit ball $\B\subset \C^N$ and let $\mu \in \CH_V$. In $\B$, there is 
a holomorphic differential operator $Q$, a holomorphic $n$-form $\vartheta$, and a holomorphic function
$h$ with $V$-polar zero set, such that
\begin{equation}\label{DLrep}
\mu. (\varphi \wedge dz)=\lim_{\epsilon \to 0^+}
\int_{V\cap \{|h|>\epsilon\}} \frac{Q(\varphi)\wedge \vartheta}{h}, \quad \varphi \in \mathscr{D}_{0,n}(\B).
\end{equation} 
\end{proposition}
\begin{proof}
Let $y \in V$ and assume that we have local coordinates $w=(w';w'')=(w_1,\ldots,w_P;w_{P+1},\ldots,w_N)$
so that $V=\{w'=0\}$ close to $y$.
If $1\leq j \leq P$ we have by 2) in Definition \ref{CHdef} that $\bar{w}_j\mu=0$, and so, by 3),
we get $d\bar{w}_j\wedge \mu = \debar (\bar{w}_j\mu)=0$.
It follows, for any function $\phi$ with support close to $y$, that
$\mu. (\phi \,d\bar{w}_I\wedge dw)=0$ if $d\bar{w}_I\neq \pm d\bar{w}''$.
Let $\Pi\colon \C^N_w \to \C^n_{w''}$ be the standard projection and define
\begin{equation*}
a_{\alpha}(w'')=\Pi_* (w^{\alpha}dw'\wedge \mu /\alpha !)
\end{equation*}
for $\alpha=(\alpha',0)$. Since $\mu$ is $\debar$-closed, the $a_{\alpha}$ must be holomorphic.
We claim that 
\begin{equation}\label{murep}
\mu = \frac{1}{(2\pi i)^P}\sum_{\alpha=(\alpha',0)}
a_{\alpha}(w'')\debar \frac{1}{w_P^{\alpha_P+1}}\wedge \cdots \wedge 
\debar \frac{1}{w_1^{\alpha_1+1}}
\end{equation}
close to $y$, where the sum ranges over $\alpha$ with $|\alpha|$ less than or equal to the order, $M$, of 
$\mu$ on $\bar{\B}$. 
Given the claim, the proposition easily follows for test forms with support close to $y$. In fact,
by the Poincar\`{e}-Lelong formula,
we may take $(-1)^{nP}\sum_{\alpha=(\alpha',0)}a_{\alpha}(w'')\partial^{\alpha}$ as the differential operator, 
let $\vartheta=dw''$, and $h=1$.
To prove the claim, we note that
it suffices to check it for test forms $\phi \,d\bar{w}''\wedge dw$ by the observation in the beginning of the proof.
We write $\phi$ as a Taylor sum
\begin{equation*}
\phi=\sum_{\alpha=(\alpha',0)} \frac{\partial^{|\alpha|}\phi}{\partial w^{\alpha}}(0,w'')
\frac{w^{\alpha}}{\alpha !} + \bar{\mathcal{J}}_V + \mathcal{O}(|w'|^{M+1}).
\end{equation*}  
Noting that $|w'|^{M+1}\mu=0$ and that 
$\bar{\mathcal{J}}_V\cdot \mu=0$,
by 2) in Definition \ref{CHdef}, the claim now follows from the definition of the $a_{\alpha}$ and a 
simple computation.

To obtain global $Q$, $\vartheta$, and $h$ we proceed as follows. We choose $H_1,\ldots,H_P\in \mathcal{J}_V(\B)$
and coordinates $z=(z',z'')$ for $\B$ such that $\tilde{h}:=\det (\partial H/\partial z')$ is generically
non-vanishing on every component of $V$, i.e., $\tilde{h}^{-1}(0)$ is $V$-polar. 
Outside $\{\tilde{h}=0\}$ we can then make the change of variables
$w=(w',w''):=(H,z'')$. Note however that $w_j\in \mathcal{O}(\B)$ for all $j$. Outside $\{\tilde{h}=0\}$
we thus have a realization \eqref{murep} of $\mu$ with $a_{\alpha}\in \mathcal{O}(\B)$.
Moreover, since $\partial /\partial w'=  \ ^t(dH/dz')^{-1} \partial /\partial z'$, we see that 
\begin{equation*}
Q':=(-1)^{nP}\tilde{h}^k \sum_{\alpha=(\alpha',0)}a_{\alpha}(w'')\partial_{w'}^{\alpha}
\end{equation*}
is a holomorphic differential operator in $\B$ if $k$ is large enough; recall that $|\alpha|\leq M$. 
For large enough $\ell$ we then define the holomorphic differential operator 
$Q$ by $Q(\phi)=\tilde{h}^{\ell}Q'(\phi/\tilde{h})$. Letting $h=\tilde{h}^{k+\ell}$ and $\vartheta=dz''$,
the formula \eqref{DLrep} then follows from \eqref{murep} if $\varphi$ has support outside $\{h=0\}$.
But both $\mu$ and the current defined by the right hand side of \eqref{DLrep}
has the Standard Extension Property, by 1) in Definition \ref{CHdef} and the comment after the proof of 
Lemma \ref{seplemma} respectively, and so the proposition follows. 
\end{proof}

This proposition makes it possible to divide Coleff-Herrera currents by holomorphic functions. 
Let $\mu\in \CH_V$ and let $f$ be a holomorphic function such that $V\setminus f^{-1}(0)$ is dense in $V$.
Given a local representation \eqref{DLrep} of $\mu$ we put
\begin{equation}\label{mugenomf}
\frac{1}{f}\mu. (\varphi\wedge dz) = \lim_{\epsilon \to 0^+} 
\int_{V\cap \{|hf|>\epsilon\}} \frac{Q(\varphi/f)\wedge \vartheta}{h}.
\end{equation}
It is clear that $(1/f)\mu \in \CH_V[*f^{-1}(0)]$. On the other hand, if $\gamma\in \CH_V[*f^{-1}(0)]$, then
(at least locally) for some large $k$ we have $\tau=f^k\gamma\in \CH_V$. Thus, $\gamma=(1/f^k)\tau$ for some 
$\tau\in \CH_V$. It follows that we have representations \eqref{DLrep} also for currents $\gamma \in \CH_V[*S]$ if 
$V\setminus S$ is dense in $V$ and that $(1/f)\gamma$ is defined. From Lemma \ref{seplemma} and the 
technique of its proof it follows that 
\begin{equation}\label{divekv1}
\frac{1}{f}\mu = \lim_{\epsilon \to 0^+}\frac{\chi(|f|^2/\epsilon)}{f}\mu =
\frac{|f|^{2\lambda}}{f}\mu \Big|_{\lambda=0}, \quad \mu\in \CH_V[*S],
\end{equation}
cf.\ also \eqref{Qchi}.

\begin{proposition}\label{destreckprop}
Let $S\subseteq X$ be a hypersurface such that $V\setminus S$ is dense in $V$. Then
\begin{equation*}
\debar \colon \CH_V[*S] \to \CH_{V\cap S}.
\end{equation*}
\end{proposition} 

\begin{remark}
This mapping actually fits into a long exact sequence, see, e.g., \cite{JEBabel}. 
In particular, if $S$ is $V$-polar, then
\begin{equation*}
0\to \CH_V \hookrightarrow \CH_V[*S] \stackrel{\debar}{\longrightarrow} \CH_{V\cap S} \longrightarrow
\mathcal{H}^{P+1}_{[V]}(\mathcal{O}_X) \to 0
\end{equation*}
is exact. Here, $\mathcal{H}^{P+1}_{[V]}(\mathcal{O}_X)$ is (isomorphic to) the cohomology group 
$\textrm{Ker}_{\debar}(\mathcal{C}_V^{0,P+1})/\debar (\mathcal{C}_V^{0,P})$, where
$\mathcal{C}_V^{0,*}$ are the currents on $X$ of bidegree $(0,*)$ with support contained in $V$. If,
in addition, $V$ is Cohen-Macaulay then this group vanishes and the mapping of Proposition \ref{destreckprop}
becomes surjective.
\end{remark}

\begin{proof}
We will start by indicating how to prove the following 

\smallskip

{\em Claim:} Let $f,g,h$ be holomorphic functions such that $h^{-1}(0)$ is $V$-polar and $V\setminus f^{-1}(0)$
is dense in $V$. Then
\begin{equation}\label{dekv1}
\lambda \mapsto \frac{\debar |f|^{2\lambda}}{f^{\ell}}\wedge \frac{1}{h}[V] \quad \textrm{has an analytic continuation
as a current},
\end{equation}
\begin{equation}\label{dekv2}
\frac{\chi(|g|^2/\epsilon)}{g^k}
\frac{\debar |f|^{2\lambda}}{f^{\ell}}\wedge \frac{1}{h}[V]\Big|_{\lambda=0} \to 0, \quad \textrm{as} \quad
\epsilon\to 0^+,
\end{equation}
if $\chi\in C^{\infty}([0,\infty])$ and vanishes both close to $0$ and $\infty$. Moreover,
if $(f,g)$ defines a complete intersection on $V$ and $\chi$ is a smooth regularization of the 
characteristic function of $[1,\infty)$ then
\begin{equation}\label{dekv3}
\chi(|g|^2/\epsilon)\frac{\debar |f|^{2\lambda}}{f^{\ell}}\wedge \frac{1}{h}[V]\Big|_{\lambda=0} \to 
\frac{\debar |f|^{2\lambda}}{f^{\ell}}\wedge \frac{1}{h}[V]\Big|_{\lambda=0}, 
\end{equation}
as $\epsilon \to 0^+$. If instead $g\in \mathcal{J}_{V\cap f^{-1}(0)}$, then
\begin{equation}\label{dekv4}
\bar{g}\frac{\debar |f|^{2\lambda}}{f^{\ell}}\wedge \frac{1}{h}[V]\Big|_{\lambda=0} \to 0.
\end{equation}

To prove \eqref{dekv1}, one computes in a resolution $\pi\colon \mathcal{V}\to V$ such that 
$\{\pi^*f \cdot \pi^*h=0\}$ has normal crossings in the manifold $\mathcal{V}$
and one chooses, preferably, local coordinates, $x$, such that 
$\pi^*h=x^{\beta}$ and $\pi^*f=ux^{\alpha}$, where $u$ is holomorphic and invertible.
To prove \eqref{dekv2} and \eqref{dekv4} one proceeds similarly; one first computes 
$(\debar |f|^{2\lambda}/f^{\ell})\wedge (1/h)[V]|_{\lambda=0}$ in a resolution
$\pi\colon \mathcal{V}\to V$ such that 
$\{\pi^*f \cdot \pi^*h \cdot \pi^*g=0\}$ has normal crossings, preferably using coordinates
such that $\pi^*f=ux^{\alpha}$, $\pi^*h=vx^{\beta}$, and $\pi^*g=x^{\gamma}$.
Then it is not too hard to verify \eqref{dekv2} and \eqref{dekv4}.
It is a bit more delicate to prove \eqref{dekv3} since the assumption about complete intersection 
has to be used properly. Let $\varphi$ be a $(n,n-1)$-test form in the base space $X$. 
On a resolution manifold $\mathcal{V}$ as the one above, one chooses an atlas of 
local coordinates with the properties stated above, and moreover, so that 
$\pi^*\varphi = \phi d\bar{x}'\wedge dx$, where $d\bar{x}'=d\bar{x}_2\wedge \cdots \wedge d\bar{x}_n$.
The trick is now to show that $(\debar |f|^{2\lambda}/f^{\ell})\wedge (1/h)[V]. \varphi|_{\lambda=0}$
equals
\begin{equation}\label{dekv5}
\sum_{\stackrel{x_1 \mid \, x^{\alpha}}{x_1 \nmid \, x^{\gamma}}}
\int \frac{\debar |x^{\alpha}|^{2\lambda} |x^{\beta}|^{2s}}{x^{\ell \alpha +\beta}} \wedge
\frac{|u|^{2\lambda} |v|^{2s}}{uv}\rho_x \, \phi d\bar{x}'\wedge dx\Big|_{s=0}\, \Big|_{\lambda=0},
\end{equation}
where $\{\rho_x\}$ is a partition of unity on $\mathcal{V}$. That is, that only charts on $\mathcal{V}$
such that $x_1 \mid \, \pi^*f$ and $x_1 \nmid \, \pi^*g$ contribute. That $x_1$ has to divide $\pi^*f$
is obvious. If, in addition, $x_1$ divides $\pi^*g$, then $\{x_1=0\}\subseteq \pi^{-1} \{f=g=0\}\cap V$.
Since $\{f=g=0\}\cap V$ has dimension $n-2$ it follows that any anti-holomorphic $n-1$-form in $X$
has a vanishing pullback to $\{f=g=0\}\cap V$. Thus, $\pi^*\varphi$ has a vanishing pullback to $\{x_1=0\}$.
It follows that $\phi=\bar{x}_1\tilde{\phi}$ for some smooth $\tilde{\phi}$. Using this one easily shows that 
charts where $x_1$ divides both $\pi^*f$ and $\pi^*g$ do not contribute. 
With this in mind it is not very difficult to show that 
$\chi(|g|^2/\epsilon)(\debar |f|^{2\lambda}/f^{\ell})\wedge (1/h)[V]. \varphi|_{\lambda=0}$ tends to
\eqref{dekv5} as $\epsilon \to 0^+$, which then proves \eqref{dekv3}.

\smallskip

It is now easy to prove the proposition. Let $\mu\in \CH_V[*S]$; it is a local problem to show that 
$\debar \mu \in \CH_{V\cap S}$. Choose a holomorphic function $f$ and a $\tau \in \CH_V$ such that
$S=f^{-1}(0)$ and $\mu = (1/f)\tau$. From \eqref{divekv1} we have $\mu=(|f|^{2\lambda}/f)\tau |_{\lambda=0}$,
and for $\mathfrak{Re} \, \lambda >> 1$ we have $\debar ((|f|^{2\lambda}/f)\tau)=(\debar |f|^{2\lambda}/f)\tau$.
From \eqref{dekv1}, the last expression also has an analytic continuation, and so 
$\debar \mu=(\debar |f|^{2\lambda}/f)\tau |_{\lambda=0}$. Using a representation \eqref{DLrep} of $\tau$ and 
\eqref{dekv1} and \eqref{dekv4} one easily sees that $\bar{g}\debar \mu=0$ if $g\in \mathcal{J}_{V\cap S}$ 
(recall that $S=f^{-1}(0)$). Similarly, if $\tau$ is represented by \eqref{DLrep} it follows from 
\eqref{dekv1}, \eqref{dekv2}, and \eqref{dekv3} that $\chi(|g|^2/\epsilon)\debar \mu \to \debar \mu$, 
as $\epsilon \to 0^+$; i.e., that $\debar\mu$ has the Standard Extension Property.
\end{proof}

\begin{definition}\label{CHprod}
Let $\mu \in \CH_V[*S]$ and let $f$ be a holomorphic function such that $V\setminus f^{-1}(0)$ is dense in $V$
and $V\cap S \setminus f^{-1}(0)$ is dense in $V\cap S$.
We define $\debar (1/f)\wedge \mu$ by 
\begin{equation*}
\debar\frac{1}{f}\wedge \mu = \debar (\frac{1}{f}\mu) - \frac{1}{f}\debar \mu.
\end{equation*}
\end{definition}

That this definition makes sense follows from Proposition \ref{destreckprop}.
It is intuitively clear that $\debar (1/f)\wedge \mu \in \CH_{V\cap f^{-1}(0)}[*S]$ but 
it is not immediate from the definition. However, letting $\mu =(1/g)\tau$, where $\tau\in \CH_V$
and $g^{-1}(0)=S$,
we get from Proposition \ref{nyckelprop} that
\begin{equation*}
\big| \debar\frac{\chi(|f|^2/\epsilon_1)}{f} \wedge \frac{\chi(|g|^2/\epsilon_2)}{g}\tau -
\debar \frac{1}{f}\wedge \frac{1}{g}\tau \big| \hspace{4cm}
\end{equation*}
\begin{equation*}
\hspace{2.5cm}\leq C \epsilon_2^{\omega}+ \big| \debar\frac{\chi(|f|^2/\epsilon_1)}{f} \wedge \mu -
\debar \frac{1}{f}\wedge \mu \big| 
\end{equation*}
\begin{equation*}
\hspace{3cm} \leq C \epsilon_2^{\omega} +\big| \debar(\frac{\chi(|f|^2/\epsilon_1)}{f} \wedge \mu) -
\debar (\frac{1}{f}\wedge \mu) \big|
\end{equation*}
\begin{equation*}
\hspace{2cm}+
\big|\frac{\chi(|f|^2/\epsilon_1)}{f} \debar\mu -
\frac{1}{f}\debar \mu \big| \lesssim \epsilon_1^{\omega} + \epsilon_2^{\omega},
\end{equation*}
where $\chi$ is a smooth regularization of the characteristic function of $[1,\infty)$. 
Thus, $(\debar\chi(|f|^2/\epsilon_1)/f)\wedge (\chi(|g|^2/\epsilon_2))\tau$ converges 
unrestrictedly to $\debar (1/f)\wedge \mu$. First letting $\epsilon_2\to 0$ and then letting 
$\epsilon_1\to 0$ we then see that 
$\debar (1/f)\wedge \mu =(1/g)\debar ((1/f)\tau)\in \CH_{V\cap f^{-1}(0)}[*S]$ by Proposition \ref{destreckprop}.

\medskip

If $\mu\in \CH_V$ and
$f=(f_1,\ldots,f_q)\colon X \to \C^q$ is holomorphic such that $(f_1,\ldots,f_p,f_j)$ defines 
a complete intersection on $V$ for $p+1\leq j \leq q$ we have thus given a meaning to 
\begin{equation*}
\debar\frac{1}{f_1}\wedge \cdots \wedge \debar \frac{1}{f_p}\cdot \frac{1}{f_{p+1}}\cdots \frac{1}{f_q}\wedge \mu.
\end{equation*}
It follows from Theorem \ref{huvudsats} that this product, apart from being alternating in $f_1,\ldots,f_p$
as it should, is independent of the ordering of the tuple $f$ and, moreover, that it coincides with the definition
of Coleff-Herrera and Passare.

\section{The key proposition}
In this section we prove the key proposition needed to prove our main theorem.
The proof of the proposition relies on a Whitney type division lemma for the pullback of anti-holomorphic forms
through modifications. This lemma appear also in \cite{hasamArkiv}.

Throughout this section our considerations are local;
$X=\B$ is the unit ball in $\C^N$ and $V$ is an analytic set of pure dimension $n$
(and codimension $P=N-n$)
defined in a neighborhood of $\B$. 

\begin{proposition}\label{nyckelprop}
Let $V\subseteq \B$ be an analytic set of pure dimension $n$, $S\subset \B$ a $V$-polar set, and
$f=(f_1,\ldots,f_q)\colon \B \rightarrow \C^q$ a holomorphic mapping such that 
$(f_1,\ldots,f_p,f_j)$
defines a complete intersection on $V$ for all $j=p+1,\ldots,q$. 
Let also $\chi_j$, $1\leq j \leq q$, be smooth on $[0,\infty]$ and vanish to order $\ell_j$ at $0$. 
Then for any $\mu \in \CH_V[*S]$ and $\varphi\in \mathscr{D}_{0,n-p}(\B)$ we have 
\begin{equation*}
\left|\frac{\debar \chi_1^{\epsilon}\wedge \cdots \wedge \debar 
\chi_p^{\epsilon} \chi_{p+1}^{\epsilon} \cdots \chi_{q-1}^{\epsilon}}{f_1^{\ell_1}\cdots f_{q-1}^{\ell_{q-1}}}
\big(\frac{\chi_q^{\epsilon}}{f_q^{\ell_q}} - \frac{1}{f_q^{\ell_q}}\big)\wedge \mu . (\varphi\wedge dz) \right| 
\leq \hspace{1cm}
\end{equation*}
\begin{equation*}
\hspace{7cm} \leq C \|\varphi\|_{M} \epsilon_{q}^{\omega}, 
\end{equation*}
where $\chi_j^{\epsilon}=\chi_j(|f_j|^2/\epsilon_j)$, $M$ and $\omega$ are positive constants that 
only depend on $f$ and $\textrm{Supp}(\varphi)$, while the positive constant $C$ also might depend
on the $C^M$-norm of the $\chi_j$-functions.
\end{proposition}

\begin{proof}
We fix a representation \eqref{DLrep} of $\mu$ (or rather its ``analytic continuation'' counterpart) and write
\begin{equation}\label{steg1}
\mathcal{I}_{f,\mu}^{\varphi}(\epsilon):= \frac{\debar \chi_1^{\epsilon}\wedge \cdots \wedge \debar 
\chi_p^{\epsilon} \chi_{p+1}^{\epsilon} \cdots \chi_{q}^{\epsilon}}{f_1^{\ell_1}\cdots f_{q}^{\ell_q}}
\wedge \mu. (\varphi \wedge dz) \hspace{2cm}
\end{equation}
\begin{equation*}
\hspace{1cm} =\int_{V} \frac{|h|^{2\lambda}}{h}
Q\big( \frac{\debar \chi_1^{\epsilon}\wedge \cdots \wedge \debar 
\chi_p^{\epsilon} \chi_{p+1}^{\epsilon} \cdots \chi_{q}^{\epsilon}}{f_1^{\ell_1}\cdots f_q^{\ell_q}}\wedge
\varphi \big)\wedge \vartheta \,\Big|_{\lambda=0}.
\end{equation*}
Since expressions like $\chi^{\epsilon}/f^{\ell}$ essentially are invariant under holomorphic differential
operators, cf.\ \eqref{invarians}, and since $\debar$ commutes with such operators, 
the right hand side of \eqref{steg1} is, by Leibniz' rule, a finite sum of integrals
of the same kind but with the holomorphic differential operator $Q$ omitted. We can therefore ignore $Q$
in the computations below.

\smallskip

By Hironaka's theorem, e.g., formulated as in \cite{Hiro} and \cite{Atiyah}, one can find, first an
$n$-dimensional complex manifold $\tilde{V}$ and a proper holomorphic map $\pi_1\colon \tilde{V}\to V$
that defines a biholomorphism outside $V_{sing}$, and then (at least locally on $\tilde{V}$) a further
$n$-dimensional complex manifold $\mathcal{V}$ and a proper holomorphic map 
$\pi_2\colon \mathcal{V}\to \tilde{V}$ such that 
$\mathcal{Z}:=\pi_2^{-1}(\pi_1^*h\cdot \pi_1^*f_1\cdots \pi_1^*f_q=0)$ has normal crossings and 
$\pi_2$ is a biholomorphism outside $\mathcal{Z}$.
Put $\pi=\pi_1\circ \pi_2$ and denote the pull-back under $\pi$ by $\hat{\cdot}$, e.g., $\hat{h}=\pi^*h$.
We choose a (sufficiently fine) finite partition of unity $\{\rho_j\}$ on $\textrm{Supp}(\hat{\varphi})$
and local charts on the $\textrm{Supp}(\rho_j)$ such that $\hat{h},\hat{f}_1,\ldots,\hat{f}_q$ are 
monomials times invertible holomorphic functions. The right hand side of \eqref{steg1}, (recall that we 
may ignore $Q$), is therefore equal to 
\begin{equation}\label{steg3}
\sum_i \int_{\mathcal{V}} \frac{|\hat{h}|^{2\lambda}}{\hat{h}}
\frac{\debar \chi_1^{\epsilon}\wedge \cdots \wedge \debar 
\chi_p^{\epsilon} \chi_{p+1}^{\epsilon} \cdots \chi_{q}^{\epsilon}}{
\hat{f}_1^{\ell_1}\cdots \hat{f}_q^{\ell_q}}\wedge
\hat{\varphi} \wedge \hat{\vartheta}\rho_i \,\Big|_{\lambda=0}.
\end{equation}
Moreover, we may assume that $\varphi$ is of the form $\varphi_Id\bar{z}_I$, $|I|=n-p$, and so we can write 
$\hat{\varphi}=\eta \cdot \phi_1$, where $\eta=\hat{\varphi}_I\in \mathscr{D}_{0,0}(\mathcal{V})$
and $\phi_1=\widehat{d\bar{z}}_I$ is an anti-holomorphic $n-p$-form on $\mathcal{V}$.
We now consider one term of \eqref{steg3}, we drop 
the subscript $i$ from $\rho_i$, and we put $\phi_2:=\eta \hat{\vartheta} \rho$.
In a neighborhood of $\textrm{Supp}(\rho)$ we have local coordinates $x$ such that
$\hat{f}_j=u_jx^{\alpha(j)}$, where $u_j$ are invertible and holomorphic. 
We let $m$ be the number of vectors in a maximal linearly independent
subset of $\{\alpha(1),\ldots,\alpha(p)\}$, and we assume for notational convenience that 
$\alpha(1),\ldots,\alpha(m)$ are linearly independent. As in \cite{PCrelle}, p.\ 46, we can define new coordinates,
still denoted $x$, so that $u_1=\cdots =u_m=1$. For $m+1\leq j \leq p$ we write
$\debar \chi_j^{\epsilon}=\tilde{\chi}_j^{\epsilon}\cdot (d\bar{x}^{\alpha(j)}/\bar{x}^{\alpha(j)}+d\bar{u}_j/\bar{u}_j)$,
where $\tilde{\chi}_j(t)=t\chi'_j(t)$ are smooth on $[0,\infty]$, vanish to order $\ell_j$ at $0$, and map $\infty$ to $0$. 
We will omit the tildes in the computations below, and hence,
with abuse of notation, the term of \eqref{steg3} under consideration can be written
\begin{equation}\label{steg4}
\int_{\mathcal{V}} \frac{|\hat{h}|^{2\lambda}}{\hat{h}}
\frac{\debar \chi_1^{\epsilon}\wedge \cdots \wedge \debar 
\chi_m^{\epsilon} \chi_{m+1}^{\epsilon} \cdots \chi_{q}^{\epsilon}}{
\hat{f}_1^{\ell_1}\cdots \hat{f}_q^{\ell_q}}\bigwedge_{m+1}^p
\big(\frac{d\bar{x}^{\alpha(j)}}{\bar{x}^{\alpha(j)}} + 
\frac{d\bar{u}_j}{\bar{u}_j}\big)
\phi_1\wedge \phi_2\,\Big|_{\lambda=0} 
\end{equation}
From exterior algebra it 
follows that
$d\bar{x}^{\alpha(1)}\wedge \cdots \wedge d\bar{x}^{\alpha(m)}\wedge d\bar{x}^{\alpha(j)}=0$ if
$m+1\leq j \leq p$ since $\alpha(1),\ldots,\alpha(m),\alpha(j)$ are linearly dependent.
Thus, since $\debar \chi_1^{\epsilon}\wedge \cdots \wedge \debar \chi_m^{\epsilon}$ is proportional
to $d\bar{x}^{\alpha(1)}\wedge \cdots \wedge d\bar{x}^{\alpha(m)}$, we may erase the factors
$d\bar{x}^{\alpha(j)}/\bar{x}^{\alpha(j)}$, $m+1\leq j \leq p$ from \eqref{steg4}. 
We now let $\mathcal{K}$ be the set of indices $i$ such that $x_i$ divides 
some $x^{\alpha(j)}$ with $p+1 \leq j \leq q$ and we apply Lemma \ref{divlemma} with in data $\mathcal{K}$ and 
$d\bar{z}_I$. We find that we may replace $d\bar{u}_{m+1}\wedge \cdots \wedge d\bar{u}_p\wedge \phi_1$
by an anti-holomorphic form $\xi=\sum_{|J|=n-m}\xi_Jd\bar{x}_J$, which has the property
that each $\xi_J$ is divisible by all $\bar{x}_i$, $i\in \mathcal{K}$, without affecting the 
integral \eqref{steg4}. We may 
of course assume that $\xi$ consists of one term only, and for notational convenience we assume that
$\xi=\xi'd\bar{x}_{m+1}\wedge\cdots \wedge d\bar{x}_n$. We assume, also for simplicity, that 
$\mathcal{K}\setminus \{m+1,\ldots,n\}=\{k+1,\ldots,m\}$ so that $\xi'$ may be written
$\bar{x}_{k+1}\cdots \bar{x}_m\xi''=:\bar{x}^{{\bf 1}_{k}^m}\xi''$ for some anti-holomorphic function $\xi''$.
We can now re-write the integral \eqref{steg4} as 
\begin{equation}\label{steg5}
\int_{\mathcal{V}} \frac{|\hat{h}|^{2\lambda}}{\hat{h}}
\frac{\debar \chi_1^{\epsilon}\wedge \cdots \wedge \debar 
\chi_m^{\epsilon} \chi_{m+1}^{\epsilon} \cdots \chi_{q}^{\epsilon}}{
\hat{f}_1^{\ell_1}\cdots \hat{f}_q^{\ell_q}}
\frac{\bar{x}^{{\bf 1}_{k}^m}\xi''\bigwedge_{m+1}^nd\bar{x}_j\wedge \phi_2}{\bar{u}_{m+1}\cdots \bar{u}_p} 
\,\Big|_{\lambda=0}. 
\end{equation}

Now consider the form
$\debar \chi_1^{\epsilon}\wedge \cdots \wedge \debar \chi_m^{\epsilon}\wedge_{m+1}^n d\bar{x}_{j}$.
We write $\debar=\debar'+\debar''$, where $\debar'$ differentiates with respect to the variables
$x'=(x_1,\ldots,x_k)$ and 
$\debar''$ differentiates with respect to the variables $x''=(x_{k+1},\ldots,x_n)$, and we compute:
\begin{equation}\label{super4}
\debar \chi_1^{\epsilon}\wedge \cdots \wedge \debar \chi_m^{\epsilon}\bigwedge_{m+1}^n d\bar{x}_{j}=
(\debar' \chi_1^{\epsilon}+\debar'' \chi_1^{\epsilon})\wedge \cdots \wedge 
(\debar' \chi_m^{\epsilon}+\debar'' \chi_m^{\epsilon})\bigwedge_{m+1}^n d\bar{x}_{j}
\end{equation}
\begin{equation*}
=\sum_{\stackrel{i_1<\cdots <i_k}{i_{k+1}<\cdots < i_m}}
\mbox{sign}(j\mapsto i_j)
\debar' \chi_{i_1}^{\epsilon}\wedge \cdots \wedge \debar' \chi_{i_k}^{\epsilon}\wedge
\debar'' \chi_{i_{k+1}}^{\epsilon}\wedge \cdots \wedge \debar'' \chi_{i_m}^{\epsilon}
\bigwedge_{m+1}^n d\bar{x}_{j}.
\end{equation*}
Let us consider the first term in this sum. It equals
\begin{equation*}
\det \tilde{A} \,\debar \chi_1^{\epsilon}\wedge \cdots \wedge \debar \chi_k^{\epsilon} \wedge
\frac{d\bar{x}_{k+1}\wedge \cdots \wedge d\bar{x}_m}{\bar{x}_{k+1}\cdots \bar{x}_m}
\tilde{\chi}_{k+1}^{\epsilon}\cdots \tilde{\chi}_m^{\epsilon}
\bigwedge_{m+1}^n d\bar{x}_{j},
\end{equation*}
where $\tilde{A}$ is the $(m-k)\times (m-k)$-matrix $(\alpha(i)_{j})_{i,j=k+1}^m$ and
$\tilde{\chi}_j^{\epsilon}=(|\hat{f}_j|^2/\epsilon_j)\cdot \chi_j'(|\hat{f}_j|^2/\epsilon_j)$. (As usual, we omit the 
tildes below.)
The other terms in the sum on the right hand side of \eqref{super4} are of the same type. In particular,
each such term has $\bar{x}_{k+1}\cdots \bar{x}_m=\bar{x}^{{\bf 1}_{k}^m}$ as denominator. 
Recall also that $\hat{h}=vx^{\beta}$ with $v$ invertible and holomorphic. Substituting 
\eqref{super4} into \eqref{steg5}
we thus obtain finitely many integrals of the type
\begin{equation}\label{steg6}
\int_{\mathcal{V}} 
\frac{\debar \chi_1^{\epsilon}\wedge \cdots \wedge \debar 
\chi_k^{\epsilon}}{x^{\ell \alpha +\beta}} |vx^{\beta}|^{2\lambda}
\chi_{k+1}^{\epsilon} \cdots \chi_{q}^{\epsilon}
\bigwedge_{k+1}^nd\bar{x}_j \wedge \psi dx \, \Big|_{\lambda=0} 
\end{equation}
where $\psi dx=\xi'' \phi_2/(v u_{m+1}^{\ell_{m+1}}\cdots u_q^{\ell_q}\bar{u}_{m+1}\cdots \bar{u}_p)$
and $\ell \alpha=\sum_1^q\ell_j \alpha(j)$. Note that $\psi$ has compact support, and, perhaps after scaling,
we may assume it has support in the unit polydisc $\Delta$.

We now introduce the smoothing parameters
\begin{equation*}
t_j=|x^{\alpha(j)}|^2/\epsilon_j, \quad \textrm{for} \quad j=k+1,\ldots,q,
\end{equation*} 
and we put
\begin{equation*}
\Psi(\lambda,x,t_{k+1},\ldots,t_q) = |v|^{2\lambda}
\Pi_{k+1}^m\chi_{i}(t_{i})\cdot \Pi_{m+1}^q
\chi_{j}(t_{j}|u_{j}|^2)\cdot \psi(x).
\end{equation*}
The function $\Psi(\lambda,x,t)$ is smooth on $\C\times \C^n\times [0,\infty]^{q-k}$ and
by Lemma 6 in \cite{hasamJFA} it has the Taylor-like expansion 
\begin{eqnarray}\label{super6}
\Psi(\lambda,x,t) &=&
\sum_{K+L<\ell \alpha+\beta -{\bf 1}} x^K\bar{x}^L \Psi_{K,L}(\lambda,x,t) \\
& &
+\sum_{K+L=\ell \alpha+\beta -{\bf 1}} x^K\bar{x}^L \Psi_{K,L}(\lambda,x,t). \nonumber
\end{eqnarray}
When doing this expansion we consider $t=(t_{k+1},\ldots,t_q)$ as independent real variables and 
$\lambda$ as a parameter.
If $K+L<\ell \alpha +\beta - {\bf 1}$, the function $\Psi_{K,L}(\lambda,x,t)$ is independent 
of at least some coordinate $x_j$ and, moreover, we have the following explicit expression for the ``remainder''
part of the expansion:
\begin{equation}\label{super7}
\sum_{K+L=\ell \alpha +\beta -{\bf 1}} x^K\bar{x}^L \Psi_{K,L}(\lambda,x,t) \hspace{3cm}
\end{equation}
\begin{equation*}
\hspace{1cm} =\int_{y\in [0,1]^{n}}
\frac{({\bf 1}-y)^{\ell \alpha +\beta -{\bf 2}}}{(\ell \alpha +\beta -{\bf 2})!}
\frac{\partial^{|\ell \alpha +\beta -{\bf 1}|}}{\partial y^{\ell \alpha +\beta -{\bf 1}}}
\Psi(\lambda,y_1x_1,\ldots,y_nx_n,t)dy.
\end{equation*}
If we evaluate the smoothing parameters, i.e., let $t_j=|x^{\alpha(j)}|^2/\epsilon_j$, we have that
$|v|^{2\lambda}\chi_{k+1}^{\epsilon} \cdots \chi_q^{\epsilon} \psi= \Psi(\lambda,x,t)$ 
and we can substitute the decomposition \eqref{super6} of $\Psi$ into \eqref{steg6}.
By changing to polar coordinates and using that $\Psi_{K,L}(\lambda,x,t)$ is independent of some $x_j$
if $K+L<\ell \alpha +\beta -{\bf 1}$ it is 
not very hard to see that the first part of the expansion \eqref{super6} does not contribute;
see \cite{PCrelle} p.\ 47, 48, for details.
In polar coordinates, the integral \eqref{steg6} thus equals
\begin{equation}\label{steg7}
\int_{r\in [0,1]^n} \mathscr{J}(\lambda,r,t) r^{2\lambda\beta}
d\chi_1^{\epsilon}\wedge \cdots \wedge d\chi_k^{\epsilon}\wedge dr_{k+1}\wedge \cdots \wedge dr_n\, \Big|_{\lambda=0},
\end{equation}
where
\begin{equation*}
\mathscr{J}(\lambda,r,t)=c_{n,k}
\int_{\theta \in [0,2\pi)^n}\sum_{\stackrel{K+L=}{\ell \alpha+\beta -{\bf 1}}}
\Psi_{K,L}(\lambda,r,\theta, t)e^{i\theta \cdot (K-L-\ell \alpha-\beta +{\bf 1}_1^k})d\theta.
\end{equation*}
Here, $\chi_j^{\epsilon}=\chi_j(r^{2\alpha(j)}/\epsilon_j)$ for $j=1,\ldots,k$, 
$t_j=r^{2\alpha(j)}/\epsilon_j$ for $j=k+1,\ldots,q$, and $c_{n,k}=(-1)^{n^2}i^n2^{n-k}$.
Since the singularity has disappeared it is 
innocuous to put $\lambda=0$ in \eqref{steg7} and from now on we omit all occurrences of 
$\lambda$. We here also note the following properties of the function $\mathscr{J}(r,t)=\mathscr{J}(0,r,t)$. 
\begin{itemize}
\item[a)] $\mathscr{J}(r,t)$ is bounded on $[0,1]^n\times [0,\infty]^{q-k}$.
\item[b)] If some $t_j<\delta$ for some $k+1\leq j \leq q$ then $|\mathscr{J}(r,t)|\leq C\delta$.
\item[c)] If $t_j>1/\delta$ for some $k+1\leq j \leq p$ then
$|\mathscr{J}(r,t)|\leq C\delta$.
\item[d)] 
\begin{equation*}
|\mathscr{J}(r,t_{k+1},\ldots,t_{q-1},t_q)-
\mathscr{J}(r,t_{k+1},\ldots,t_{q-1},\infty)| \hspace{1cm}
\end{equation*}
\begin{equation*}
\hspace{5cm}
\leq
\begin{cases}
C\|\psi\|_{|\ell \alpha+\beta -{\bf 1}|} /t_q, & t_q \geq 1 \\
C\|\psi\|_{|\ell \alpha+\beta -{\bf 1}|}, & t_q \leq 1
\end{cases}
\end{equation*}
\end{itemize}
Property a) follows easily from formula \eqref{super7} and the definition of $\Psi(x,t)=\Psi(0,x,t)$.
Properties b) and c) follow by Taylor expanding $t_j \mapsto \chi_j(t_j|u_j|^2)$ 
at $0$ and $\infty$ respectively
in the definition of $\Psi(x,t)$. Property d) follows by Taylor expanding 
$t_q \mapsto \chi_j(t_q|u_q|^2)$ at $\infty$ in the definition of $\Psi(x,t)$ and inspection in the 
formula \eqref{super7}. We note also that the constant(s), $C$, depend on the 
$C^{|\ell \alpha +\beta -{\bf 1}|}$-norm
of $\chi_j$ if $u_j\neq 1$.
Now, by Fubini's theorem we may write
\eqref{steg7} as
\begin{equation}\label{steg8}
\int_{r''\in [0,1]^{n-k}}\Big(\int_{r'\in [0,1]^{k}}
\mathscr{J}(r',r'',t)d\chi_1^{\epsilon}\wedge\cdots \wedge d\chi_k^{\epsilon}\Big)
dr_{k+1}\wedge \cdots \wedge dr_n
\end{equation} 
and by property a), the modulus of the inner integral can be estimated by a constant times 
$\int_{s\in [0,\infty)^{k}}|d\chi_1(s_1)\wedge\cdots \wedge d\chi_k(s_k)| \leq C<\infty$ 
uniformly in all parameters. Thus, by Dominated Convergence, the study of possible limits of \eqref{steg7}, and hence of 
\eqref{steg1}, is reduced to the study of possible limits of the inner integral in \eqref{steg8} for fixed
$r''=(r_{k+1},\ldots,r_n)\in [0,1]^{n-k}$.
We note here also that 
$d\chi_1^{\epsilon}\wedge\cdots \wedge d\chi_k^{\epsilon}= \det A \cdot
\tilde{\chi}_1^{\epsilon}\cdots \tilde{\chi}_k^{\epsilon}dr_1\wedge \cdots \wedge dr_k/(r_1\cdots r_k)$,
where $A$ is the $k\times k$-matrix $(2\alpha(i)_j)_{i,j=1}^k$. We may thus assume that $A$ is invertible.

We want to estimate the difference 
$|\mathcal{I}_{f,\mu}^{\varphi}(\epsilon)-\lim_{\epsilon_q\rightarrow 0}
\mathcal{I}_{f,\mu}^{\varphi}(\epsilon)|$ and by our computations this far it can be estimated by a 
finite sum of terms like
\begin{equation*}
\int_{[0,1]^{n-k}}\Big(\int_{[0,1]^{k}}
|\mathscr{J}(r',r'',t)-\lim_{\epsilon_q\rightarrow 0}\mathscr{J}(r',r'',t)|\cdot
|d\chi_1^{\epsilon}\wedge\cdots \wedge d\chi_k^{\epsilon}| \Big)
dr''
\end{equation*}
\begin{equation*}
=\int_{[0,1]^{n-k}}\Big(\int_{[0,1]^{k}}
|\mathscr{J}(r',r'',t)-\mathscr{J}(r',r'',t)|_{t_q=\infty}|\cdot
|d\chi_1^{\epsilon}\wedge\cdots \wedge d\chi_k^{\epsilon}| \Big)
dr''.
\end{equation*}
The equality $\lim_{\epsilon_q\rightarrow 0}\mathscr{J}(r',r'',t)=\mathscr{J}(r',r'',t)|_{t_q=\infty}$
holds for each fixed $r''\in (0,1]^{n-k}$ because $t_q=r^{2\alpha(q)}/\epsilon_q$ and,
since $\mathcal{K}\subseteq \{k+1,\ldots,n\}$, we have 
$\alpha(q)_1=\cdots =\alpha(q)_k=0$. By property d) we can therefore estimate
$|I_{f,\mu}^{\varphi}(\epsilon)-\lim_{\epsilon_q\rightarrow 0}
I_{f,\mu}^{\varphi}(\epsilon)|$ by a finite sum of terms of the form
\begin{equation*}
C\|\psi\|_{|\ell\alpha +\beta -{\bf 1}|}\big(
\int_{\stackrel{r''\in [0,1]^{n-k}}{r^{2\alpha(q)}\leq \epsilon_q}} dr''+
\int_{\stackrel{r''\in [0,1]^{n-k}}{r^{2\alpha(q)}\geq \epsilon_q}} \epsilon_q/r^{2\alpha(q)}dr''\big) \hspace{2cm}
\end{equation*}
\begin{equation*}
\leq C\|\psi\|_{|\ell\alpha +\beta -{\bf 1}|}\epsilon_q^{\omega}.
\end{equation*}
The last estimate follows from Lemmas 9 and 10 in \cite{hasamJFA}, from which we also see that any
$\omega < 1/(2|\alpha(q)|)$ works. To conclude the proof we just have to note that $\psi$
depends continuously on $\varphi$ in $C^m$-norm if $m$ is sufficiently large.
\end{proof}

\subsection{The division lemma}
We keep the notation from the proof of Proposition \ref{nyckelprop} so that 
$\mathcal{V}\to V\subseteq \B$ is a modification, the pull-back under this map is denoted by $\hat{\cdot}$,
and $x$ are local coordinates on $\mathcal{V}$ such that $\hat{f}_j=x^{\alpha(j)}$, $1\leq j \leq m$,
and $\hat{f}_j=u_jx^{\alpha(j)}$, $m+1 \leq j \leq q$. 
We recall also that our set-up in the proof of Proposition \ref{nyckelprop}
implies that the exterior product of $\wedge_1^mdx^{\alpha(i)}$ with any $dx^{\alpha(j)}$,
$m+1\leq j\leq p$, is zero. 

\begin{lemma}\label{divlemma}
Let $\mathcal{K}\subseteq \{1,\ldots,n\}$ be any set of indices $i$ such that $x_i$ divides some 
$x^{\alpha(j)}$ with $p+1\leq j \leq q$.
If $\sigma$ is an anti-holomorphic $n-p$-form in $\B$, then one can find, in the $x$-chart on $\mathcal{V}$,  
an anti-holomorphic $n-m$-form $\xi$ that depends continuously on $\sigma$ in
any $C^k$-norm and such that
\begin{itemize}
\item[i)] $\frac{d\bar{x}_j}{\bar{x}_j}\wedge \xi$ 
is non-singular for all $j\in \mathcal{K}$, and
\item[ii)] $d\bar{x}^{\alpha(1)}\wedge\cdots\wedge d\bar{x}^{\alpha(m)}\wedge (
d\bar{u}_{m+1}\wedge\cdots\wedge d\bar{u}_p\wedge\hat{\sigma}-\xi)=0$.
\end{itemize}  
\end{lemma}

\begin{proof}
Put $\Psi=d\bar{u}_{m+1}\wedge\cdots\wedge d\bar{u}_p\wedge\hat{\sigma}$ and define
\begin{equation*}
\tilde{\xi}=\sum_{j\in \mathcal{K}}\Psi_{j}-\sum_{\stackrel{i,j\in \mathcal{K}}{i<j}}\Psi_{ij}+\cdots
+(-1)^{|\mathcal{K}|-1}\Psi_{i_1\cdots i_{|\mathcal{K}|}},
\end{equation*}
where $\Psi_{i_1\cdots i_{\ell}}$ means that we pull $\Psi$  back to $\{x_{i_1}=\cdots=x_{i_{\ell}}=0\}$
and extend trivially to $\C^n$, (i.e, $\Psi_{i_1\cdots i_{\ell}}=\tau^*\Psi$, where $\tau$ is the composition
of the standard projection $\C^n\to \Lambda=\{x_{i_1}=\cdots=x_{i_{\ell}}=0\}$ and the inclusion 
$\Lambda \hookrightarrow \C^n$).
A straight forward induction over $|\mathcal{K}|$ shows that $\xi:=\Psi-\tilde{\xi}$
satisfies i); see, e.g., \cite{hasamArkiv}. To see that $\xi$ satisfies ii), consider 
a $\Psi_{i_1\cdots i_{\ell}}$. Let $L\subseteq \{1,\ldots,p\}$ be the set of indices $j$ such that no $x_{i_{k}}$,
$1\leq k \leq \ell$, divides $\hat{f}_j$ and write $L=L'\cup L''$, where $L'=\{j\in L;\, j\leq m\}$
and $L''=\{j\in L;\, m+1\leq j\leq p\}$. 

Now, the variety $\{x_{i_1}=\cdots=x_{i_{\ell}}=0\}$ lies
in the zero set of all the $\hat{f}_j$ with $j\in \{1,\ldots,p\}\setminus L$ by the definition of $L$, 
and moreover, it is contained
in the zero set of (at least) some $\hat{f}_{\nu}$ with $p+1\leq \nu \leq q$ since the $x_{i_j}$
are in $\mathcal{K}$. Thus, $\{x_{i_1}=\cdots=x_{i_{\ell}}=0\}$ is contained in the preimage
of a subvariety of $V$ of dimension at most $n-p+|L|-1$.
The form $\wedge_{j\in L}d\bar{f}_j\wedge \sigma$
has degree $n-p+|L|$ and so its pullback to this variety must vanish. Hence,
\begin{equation*}
\widehat{\bigwedge_{j\in L}d\bar{f}_j\wedge \sigma}=
\bigwedge_{j\in L}d\hat{\bar{f}}_j \wedge \hat{\sigma}=
\bigwedge_{i\in L'}d\bar{x}^{\alpha(i)} \wedge \bigwedge_{j\in L''}d(\bar{u}_j\bar{x}^{\alpha(j)}) \wedge \hat{\sigma}
\end{equation*}
has a vanishing pullback to $\{x_{i_1}=\cdots=x_{i_{\ell}}=0\}$.
But this means that 
\begin{equation*}
\bar{x}^{\sum_{i\in L''}\alpha(i)}\bigwedge_{j\in L'} d\bar{x}^{\alpha(j)}\wedge
(\bigwedge_{k\in L''}d\bar{u}_{k}\wedge \hat{\sigma})_{i_1\cdots i_{\ell}}
\end{equation*}
\begin{equation*}
+\bigwedge_{\iota\in L'} d\bar{x}^{\alpha(\iota)}\wedge
\sum_{\nu\in L''}d\bar{x}^{\alpha(\nu)}\wedge \tau_{\nu}=0,
\end{equation*}
where the first term arises when no differential hits $\bar{x}^{\alpha(j)}$, $j\in L''$. Taking 
the exterior product with $\wedge_{j\notin L''}(d\bar{u}_j)_{i_1\cdots i_{\ell}}$ we obtain
\begin{equation*}
\bar{x}^{\sum_{i\in L''}\alpha(i)}\bigwedge_{j\in L'} d\bar{x}^{\alpha(j)}\wedge
\Psi_{i_1\cdots i_{\ell}}+\bigwedge_{\iota\in L'} d\bar{x}^{\alpha(\iota)}\wedge
\sum_{\nu\in L''}d\bar{x}^{\alpha(\nu)}
\wedge\tilde{\tau}_{\nu}=0.
\end{equation*}
We now multiply this equation with the exterior product of all $d\bar{x}^{\alpha(j)}$ with 
$j\leq m$ and $j\notin L'$. Then we get $d\bar{x}^{\alpha(1)}\wedge \cdots \wedge d\bar{x}^{\alpha(m)}$ in front
of the sum and this makes all terms under the summation sign disappear by the comment just before Lemma
\ref{divlemma}. It thus follows that 
\begin{equation*}
\bar{x}^{\sum_{i\in L''}\alpha(i)}d\bar{x}^{\alpha(1)}\wedge \cdots \wedge d\bar{x}^{\alpha(m)}
\wedge \Psi_{i_1\cdots i_{\ell}}=0,
\end{equation*}
and since this holds everywhere we may remove the factor $\bar{x}^{\sum_{i\in L''}\alpha(i)}$ and conclude that 
$\xi$ has the property ii).
\end{proof}

\section{Non-characteristic restrictions}\label{ickekarsektion}
Let $\Omega\subseteq \mathbb{R}^k$ be an open set, let $u\in \mathscr{D}'(\Omega)$,
and let $M\subseteq \Omega$ be a smooth submanifold.
Let also $\mathcal{N}(M)$ be the subbundle of $T^*(\Omega)\mid_M$ of covectors that annihilate $T(M)$.
We say that $M$ is non-characteristic for $u$ if $\mathcal{N}(M)\cap WF (u)=\emptyset$, where
$WF(u)$ is the wave front set of $u$. If $M$ is non-characteristic for $u$, then there is a well
defined ``restriction'', $u|_M$, of $u$ to $M$ and moreover, if $u_{\epsilon}\to u$ is any smooth regularization
of $u$ and $i\colon M \to \Omega$ is the inclusion map, then $i^* u_{\epsilon} \to u|_M$ is a 
regularization of $u|_M$; see \cite{Hormander1} Chapter VIII.

Let $\mu\in \CH_V(X)$, where $X\subseteq \C^n$ is a domain and $V$ is a pure dimensional analytic subset.
Then, since $\mu$ generate a regular holonomic $\mathscr{D}_X$-module, \cite{JEBabel},
it follows from a deep result of Andronikof, \cite{And}, that $WF(\mu) = WF_A(\mu)$ is a 
$\C^*$-conic complex Lagrangian subset of $T^*(X)$. Thus, by the Morse-Sard theorem, there are 
``many'' non-characteristic hypersurfaces for $\mu$, e.g., in appropriate coordinates, $x$, 
centered at an arbitrary point in $X$, all 
$H_{j,s}=\{x_j=s\}$, $s\in \C$, $0<|s|<1$, are non-characteristic for $\mu$. 
Now, let $f=(f_1,\ldots,f_p)$ be a holomorphic tuple defining a complete intersection
in $X$ and let $R^f=\debar (1/f_1)\wedge \cdots \wedge \debar (1/f_p)$ 
be the Coleff-Herrera product.
\begin{theorem}\label{icke-kar}
Let $Y\subseteq X$ be a complex submanifold that is non-characteristic for $\mu$ and such that
$f_1|_Y,\ldots,f_p|_Y$ define a complete intersection on $Y$. Then
\begin{equation*}
R^f|_Y=\debar \frac{1}{f_1|_Y}\wedge \cdots \wedge \debar \frac{1}{f_p|_Y},
\end{equation*}
i.e., the Coleff-Herrera product commutes with non-characteristic restrictions.
\end{theorem}
\begin{proof}
This follows immediately from our main theorem since
\begin{equation*}
R^f|_Y=\lim_{\epsilon\to 0} 
i^* \Big( \frac{\debar \chi_1^{\epsilon}\wedge \cdots \wedge \debar \chi_p^{\epsilon}}{f_1\cdots f_p}\Big) =
\debar \frac{1}{f_1|_Y}\wedge \cdots \wedge \debar \frac{1}{f_p|_Y},
\end{equation*}
where $i\colon Y\to X$ is the inclusion.
\end{proof}

\begin{remark}
To prove this theorem it is sufficient to use Passare regularizations. In fact, one only has to 
ensure that $\epsilon \to 0$ inside Passare sectors for both of the tuples $(f_1,\ldots,f_p)$ and 
$(f_1|_Y,\ldots,f_p|_Y)$.
\end{remark}

We conclude this section with an application of Theorem \ref{icke-kar}. Formally computing the average
of $I_f^{\varphi}(t)$ (cf., the introduction) 
for $t\in \{\xi\in \mathbb{R}^p_+; \sum_j \xi_j =\epsilon\}$ one obtains
\begin{equation*}
c_p\int_{|f|^2=\epsilon}\frac{\sum_j (-1)^{j-1}\bar{f}_j \bigwedge_{k\neq j}d\bar{f}_k}{|f|^{2p}}\wedge \varphi,
\end{equation*} 
where $c_p=(-1)^{p(p-1)/2}(p-1)!/\sqrt{p}$. Using Hironaka's theorem and toric resolutions one can show that
the limit of this integral exists and defines a Bochner-Martinelli type current $R^f_{BM}\in \CH_{f^{-1}(0)}$
that can be regularized by 
\begin{equation*}
\debar \chi (|f|^2/\epsilon)\wedge c_p
\frac{\sum_j (-1)^{j-1}\bar{f}_j \bigwedge_{k\neq j}d\bar{f}_k}{|f|^{2p}},
\end{equation*}
see, e.g., \cite{PTY} and \cite{hasamJFA}. If the test form $\varphi$ is $\debar$-closed in a 
neighborhood of $f^{-1}(0)$, then $I_f^{\varphi}(t)$ is independent of sufficiently small $t$ and the 
formal computation above becomes rigorous . However, it is a non-trivial, but well known fact that actually
$R^f=R^f_{BM}$, see \cite{PTY}, \cite{matsa2}. In fact, following \cite{matsa2}, $R^f$ and $R^f_{BM}$ 
are $\nabla$-cohomologous, (see, e.g., Section \ref{bevissektion} 
for the definition of the $\nabla$-operator) and, unlike the $\debar$-operator, 
the $\nabla$-operator admits localizations, i.e., there are $\nabla$-closed ``cut off forms''. 
Thus, one can find a current $\mu$ with support close to $f^{-1}(0)$ such that $\nabla \mu = R^f-R^f_{BM}$
and this yields a current $\mu'$, supported close to $f^{-1}(0)$, such that $\debar \mu' =R^f-R^f_{BM}$.
This implies that $R^f=R^f_{BM}$ since a $\CH_{f^{-1}(0)}$-current cannot be $\debar$-exact in this sense
unless it is $0$.  

By Theorem \ref{icke-kar} we get an independent proof of the equality $R^f=R^f_{BM}$
based on induction over $n$. In the the absolute case, i.e., when $p=n$,
it is easy to verify the equality by Taylor expanding the test form at the discrete set of points of $f^{-1}(0)$; 
we may assume that $f^{-1}(0)=0$ and then both $R^f$ and $R^f_{BM}$ are annihilated by any $\bar{x}_j$ 
and we know that they coincide on holomorphic monomials.  
Theorem \ref{icke-kar} provides us with the induction step since the set of non-characteristic hyperplanes,
$H_{j,s}$, is sufficiently ample. We thus get
\begin{corollary}
The Bochner-Martinelli type current $R^f_{BM}$ is equal to the Coleff-Herrera product $R^f$ in the 
case of a complete intersection.
\end{corollary}

\begin{remark}
For the currents of Cauchy-Fantappi\`{e}-Leray type introduced in \cite{matsa} we have the same result.
\end{remark}





\section{Proof and extensions of Theorem \ref{huvudsats}}\label{bevissektion}

\begin{proof}[Proof of Theorem \ref{huvudsats}]
The proof is based on induction over $p$. The induction start, $p=0$, follows immediately 
from Proposition \ref{nyckelprop} (and an obvious induction over $q$). Now assume that 
Theorem \ref{huvudsats} is true for $p=r-1\geq 0$.
By induction over $q$, what we have to show is that 
\begin{equation}\label{bevisekv}
\Big|\frac{\debar \chi^{\epsilon}_1\wedge \cdots \wedge \debar \chi^{\epsilon}_r}{
f_1^{\ell_1}\cdots f_r^{\ell_p}}\wedge \tilde{\mu}. \varphi - 
\debar \frac{1}{f_1^{\ell_1}}\wedge \cdots \wedge \debar \frac{1}{f_r^{\ell_r}}\wedge \tilde{\mu}. \varphi \Big|
\leq C\epsilon^{\omega}\|\varphi \|_M,
\end{equation}
where $\tilde{\mu}=(1/(f_{r+1}^{\ell_{r+1}}\cdots f_q^{\ell_q}))\mu \in \CH_V[*S]$, $S:=\{f_{r+1}\cdots f_q=0\}$.
An easy set-theoretic computation shows that $(f_1,\ldots f_r)$ defines a complete intersection 
on $V\cap S$ and by Proposition \ref{destreckprop} we have $\debar \tilde{\mu} \in \CH_{V\cap S}$.
The left hand side of \eqref{bevisekv} can be estimated by
\begin{equation*}
\Big|\frac{\debar \chi^{\epsilon}_\wedge \cdots \wedge \debar \chi^{\epsilon}_{p} \chi^{\epsilon}_r}{
f_1^{\ell_1}\cdots f_r^{\ell_r}}\wedge \tilde{\mu}. \,\debar \varphi - 
\debar \frac{1}{f_1^{\ell_1}}\wedge \cdots \wedge \debar \frac{1}{f_{p}^{\ell_{p}}}\frac{1}{f_r^{\ell_r}}
\wedge \tilde{\mu}. \,\debar \varphi \Big| + \hspace{1cm}
\end{equation*}
\begin{equation*}
+\Big|\frac{\debar \chi^{\epsilon}_1\wedge \cdots \wedge \debar \chi^{\epsilon}_{p}\chi^{\epsilon}_r}{
f_1^{\ell_1}\cdots f_r^{\ell_r}}\wedge \debar \tilde{\mu}. \, \varphi - 
\debar \frac{1}{f_1^{\ell_1}}\wedge \cdots \wedge \debar \frac{1}{f_{p}^{\ell_{p}}}\frac{1}{f_r^{\ell_r}}
\wedge \debar \tilde{\mu}. \,\varphi \Big|.
\end{equation*}
By induction, the last term can be estimated by $C \epsilon^{\omega} \|\varphi \|_{M}$. The first term can, 
by Proposition \ref{nyckelprop}, be estimated by $C\epsilon_r^{\omega}\|\debar \varphi\|_{M}$ plus
\begin{equation*}
\Big|\frac{\debar \chi^{\epsilon}_\wedge \cdots \wedge \debar \chi^{\epsilon}_{p}}{
f_1^{\ell_1}\cdots f_{p}^{\ell_{p}}}\wedge \mu'. \,\debar \varphi - 
\debar \frac{1}{f_1^{\ell_1}}\wedge \cdots \wedge \debar \frac{1}{f_{p}^{\ell_{p}}}\wedge 
\mu'. \,\debar \varphi \Big|,
\end{equation*}
where $\mu'=(1/f_r^{\ell_r})\tilde{\mu}\in \CH_V[*S']$, $S'=S\cup f_r^{-1}(0)$.
Another ``integration by parts'' and the induction hypothesis finally shows that this term can be 
estimated by $C \epsilon^{\omega} \|\debar \varphi \|_{M}$ and we are done.
\end{proof}

\smallskip

Finally we present some extensions of our main theorem. By going through the proof one verifies the following.
Let $\tilde{f}=(\tilde{f}_1,\cdots,\tilde{f}_q)$ be a holomorphic tuple such that 
$(\tilde{f}_1,\ldots,\tilde{f}_p,\tilde{f}_j)$ define a complete intersection on $V$ for all $j=p+1,\ldots,q$ 
and assume that
$\tilde{f}^{-1}(0)\supseteq f^{-1}(0)$ and that the $\chi_j$ vanish to infinite order at $0$. We may then replace 
the $\chi_j(|f_j|^2/\epsilon_j)$ in Theorem \ref{huvudsats} by $\chi_j(|\tilde{f}_j|^2/\epsilon_j)$ with the same 
conclusion; the constants $C$, $M$, and $\omega$ are unaffected.

\smallskip

Let $V$ be an analytic set of pure dimension $n$ defined in a neighborhood of $\bar{\B}\subset \C^N$;
put $\textrm{codim} (V)=P=N-n$. In a slightly smaller neighborhood of $\bar{\B}$ one can find a free resolution
\begin{equation}\label{upplosn}
0\to \mathcal{O}(E_{\nu}) \stackrel{F_{\nu}}{\longrightarrow} \cdots
\stackrel{F_2}{\longrightarrow} \mathcal{O}(E_1)
\stackrel{F_1}{\longrightarrow} \mathcal{O}(E_0)
\end{equation}
of the sheaf $\mathcal{O}_{\C^n}/\mathcal{J}_V$; the $E_j$ are trivial holomorphic vector bundles of ranks
$r_j$ with $r_0=1$ and the $F_j$ are $r_{j-1}\times r_j$-matrices of functions holomorphic in a neighborhood of 
$\bar{\B}$. Let $V_j$ be the set of points $z\in \B$ such that $F_j(z)$ does not have optimal rank.
These sets are analytic subsets of $\B$ that are independent of the choice of resolution 
of $\mathcal{O}_{\C^n}/\mathcal{J}_V$, i.e., invariants of $V$. Moreover, 
$V_j \subseteq V_{j-1} \subseteq \cdots \subseteq V_P = V_{P-1}=\cdots =V$ and since $V$ has pure dimension
in our case, Corollary 20.14 in \cite{Eis} implies that for $j>P$ one has
$V_j\subseteq V_{sing}$ and $\textrm{codim}(V_j)\geq j+1$. If $k=\textrm{max}\{j; V_j\neq \emptyset\}$
then one can find a new resolution with $\nu=k$. Let us assume that \eqref{upplosn} is such a minimal 
resolution. Let us also note that if $V$ is defined by a complete intersection, then the 
Koszul complex provides a minimal resolution.

Given Hermitian metrics on the $E_j$, Andersson-Wulcan, \cite{AW}, construct a current, $R^V$, 
whose annihilator sheaf is $\mathcal{J}_V$. This current has the form
\begin{equation*}
R^V=R_P^V +\cdots + R_{\nu}^V,
\end{equation*}
where the $R_j^V$ are $E_j$-valued $(0,j)$-currents with support in $V$ and with the Standard Extension Property
with respect to $V$. For some recent applications of $R^V$ we refer to \cite{AS} and \cite{ASS}.
If $\mathcal{J}_V$ is Cohen-Macaulay then $\nu=P$ and $R^V=R_P^V$ is $\debar$-closed, in fact, $R^V$ is then
a tuple of $\CH_V$-currents. In general, $R^V$ is not $\debar$-closed but satisfies instead $\nabla_F R^V=0$,
where $\nabla_F=\sum_j F_j - \debar$; in the case that \eqref{upplosn} is the Koszul complex
this is the $\nabla$-operator referred to in Section \ref{ickekarsektion}.  

Let $f_1,\ldots,f_q\in \mathcal{O}(\B)$ and assume that for each $\ell \geq P$
\begin{equation}\label{kommvillkor}
\textrm{codim}\, (V_{\ell}\cap \{f_1=\cdots = f_p=f_j=0\})\geq \ell +p+1, \quad \forall j\geq p+1.
\end{equation} 
Then one can define the product
\begin{equation}\label{AWprod}
\debar \frac{1}{f_1}\wedge \cdots \wedge \debar \frac{1}{f_p} \frac{1}{f_{p+1}}\cdots \frac{1}{f_q} \wedge R^V
\end{equation}
by an iterative procedure similar to the one described in Section \ref{CHsektion} and the product
has the natural suggestive commutation properties, see the end of Section 2 in \cite{AW2}. 
With our techniques one can prove
\begin{theorem}\label{AWsats}
Let $R^V$ be the current in $\B$ described above and let $(f_1,\ldots,f_q)$ be a tuple of holomorphic
functions in $\B$ that satisfies \eqref{kommvillkor}. Then, with the notation and the hypothesis on the 
$\chi$-functions from Theorem \ref{huvudsats},
we have
\begin{equation*}
\Big| \frac{\debar \chi_1^{\epsilon}\wedge \cdots \wedge \debar \chi_p^{\epsilon}
\chi_{p+1}^{\epsilon}\cdots \chi_q^{\epsilon}}{f_1 \cdots f_q} \wedge R^V. \varphi \hspace{5cm}
\end{equation*} 
\begin{equation*}
\hspace{3cm} -\debar\frac{1}{f_1}\wedge \cdots \wedge \debar\frac{1}{f_p}
\frac{1}{f_{p+1} \cdots f_q}
\wedge R^V. \varphi \Big| \leq C \|\varphi\|_M \epsilon^{\omega},
\end{equation*}
for $\oplus_j E_j^*$-valued test forms $\varphi$ in $\B$.
\end{theorem}
To prove this we need Proposition \ref{nyckelprop} with $\mu$ replaced by $R^V_j$, $j=P,\ldots,\nu$.
Then Theorem \ref{AWsats} follows, e.g., by a double induction over $p$ and $q$ and using that 
we already know that the product \eqref{AWprod} has nice commutation properties.
In the induction steps one uses that the involved $\debar$-operators may be replaced by 
$-\nabla_F$-operators.

Let us indicate how to prove the required analogue of Proposition \ref{nyckelprop}.
First of all, $R^V_P$ has an integral representation similar to \eqref{DLrep}; see, e.g.,
Proposition 2.1 in \cite{ASS} for details. Using this, the proof of Proposition \ref{nyckelprop} goes 
through with $\mu$ replaced by $R^V_P$. For $R^V_j$ with $j>P$ one uses that $R_j^V=a_j\wedge R^V_{j-1}$,
where $a_j$ is a $\textrm{Hom}(E_{j-1}, E_j)$-valued $(0,1)$-form that is smooth outside $V_j$ and moreover
has the property that in a suitable resolution $\pi \colon \mathcal{V} \to V$, $\pi^* a_j$ is a smooth
$(0,1)$-form $b_j$ divided by a holomorphic monomial; again, see Proposition 2.1 in \cite{ASS}.

Now, for simplicity, consider $R^V_{P+1}=a_{P+1}\wedge R^V_P$ and choose such a resolution
$\pi \colon \mathcal{V} \to V$ that, apart from having the properties in the proof of Proposition 
\ref{nyckelprop}, also is such that the preimage of $V_{P+1}$ is a normal crossings divisor.
The proof then goes through if we establish a somewhat more general division lemma, namely
(see the proof of Proposition \ref{nyckelprop} for the notation):

\smallskip

{\em One can replace 
$\Psi:= d\bar{u}_{m+1}\wedge \cdots \wedge d\bar{u}_p\wedge b_{P+1}\wedge \pi^*(d\bar{z}_I)$,
$|I|=n-p-1$, by a smooth form $\xi$, without affecting the integral (corresponding to \eqref{steg4}),
such that $(d\bar{x}_i/\bar{x}_i)\wedge \xi$ is $C^{r}$-smooth for an appropriate large $r$
and all $i\in \mathcal{K}$.}

\smallskip

This can be achieved as follows. Put
\begin{equation*}
\xi'=\sum_{J\subseteq \mathcal{K}}(-1)^{|J|+1}\Psi_J^{r}, \quad \textrm{where} \quad
\Psi_J^{r}:=\sum_{j\in J, k_j=0}^{r +1}
\frac{\partial^{|k|} \Psi}{\partial x_J^k}\Big|_{x_J=0} \cdot \frac{x_J^k}{k!},
\end{equation*}
cf., the beginning of the proof of Lemma \ref{divlemma}. One verifies by induction that
$\xi:=\Psi-\xi'$ satisfies $(d\bar{x}_i/\bar{x}_i)\wedge \xi \in C^{r}$ for all $i\in \mathcal{K}$.
Moreover, using \eqref{kommvillkor} for $\ell =P+1$ and the technique of the proof of Lemma
\ref{divlemma} one shows that 
$\debar \chi_1^{\epsilon}\wedge \cdots \wedge \debar \chi_m^{\epsilon}\wedge \xi'=0$
so that $\Psi$ may be replaced by $\xi$ without affecting the integral.

\end{document}